\documentclass[
12pt,
fleqn]
{article}
\usepackage[english]{babel}
\usepackage[T1]{fontenc}
\usepackage{amsmath}
\usepackage{xcolor}
\usepackage{hyperref}
\usepackage{amssymb}
\usepackage{amsthm}

\usepackage[title]{appendix}

\newtheorem{theorem}{Theorem}[section]
\newtheorem{proposition}[theorem]{Proposition}
\newtheorem{lemma}[theorem]{Lemma}
\newtheorem{Corollary}[theorem]{Corollary}

\theoremstyle{definition}

\newtheorem{remark}[theorem]{Remark}

\usepackage{graphicx}
\DeclareMathOperator{\dv}{div}
\DeclareMathOperator{\Tr}{tr}
\DeclareMathOperator{\tr}{tr}
\DeclareMathOperator{\diam}{diam}
\DeclareMathOperator{\dist}{dist}

\DeclareMathOperator{\interior}{int}

\DeclareMathOperator{\cl}{clos}

\newcommand{\R}{{\mathbb R}}

\newcommand{\e}{\epsilon}

\title{Average gradient localisation for degenerate elliptic equations in the plane }

\author{Thibault Lacombe\footnote{Institut de Math\'ematiques de Toulouse, UMR 5219, Universit\'e de Toulouse, CNRS, UPS
IMT, 31062 Toulouse Cedex 9, France. \tt{thibault.lacombe@math.univ-toulouse.fr}}}

\begin{document}

\maketitle
\begin{abstract}
We consider Lipschitz solutions to the possibly highly degenerate elliptic equation $ \dv G(\nabla u)=0 $ in $B_1\subset\R^2 $, for any continuous strictly monotone vector field $ G\colon\R^2\to\R^2$. We show that $u$ is either $C^1$ at $0$, or any blowup limit $v(x)=\lim \frac{u(\delta x)-u(0)}{\delta} $ along a sequence $\delta\to 0$ satisfies $ \nabla v\in \mathcal{D}\cap \mathcal{S} \text{ a.e} $. Here, $ \mathcal{D}$ and $\mathcal{S}$ can be roughly interpreted as the sets where ellipticity degenerates from below and above, that is, the symmetric parts of $ \nabla G$ and $(\nabla G)^{-1}$ have a zero eigenvalue. This is a strong indication in favor of the expected continuity of $H(\nabla u)$ for any continuous $H$ vanishing on $\mathcal{D}\cap \mathcal{S}$. In contrast with previous results in the same spirit, we do not make any assumption on the structure of $G$ besides its continuity and strict monotony.
\end{abstract}

\section{Introduction}
\subsection{Main Result}

This work focuses on the regularity of Lipschitz solutions \( u: B_1 \subset \mathbb{R}^2 \to \mathbb{R} \) to the nonlinear equation
\begin{align}
\label{equation div}
\operatorname{div}(G(\nabla u)) = 0,
\end{align}
where \( G : \mathbb{R}^2 \to \mathbb{R}^2 \) is a strictly monotone field, that is
\[
\langle G(\xi) - G(\zeta), \xi - \zeta \rangle > 0 \quad \forall \xi \neq \zeta \in \mathbb{R}^2.
\]

As shown in \cite{LacLam}, a sufficient condition for Lipschitz solutions of \eqref{equation div} to belong to \( C^1(B_1) \) is, roughly speaking, the finiteness of the set of points $\xi \in \R^2$ where the usual ellipticity condition 
\begin{equation}      
\label{eq: Elliptic inequality} 
\lambda \vert \zeta \vert^2 \leq \langle \nabla G(\xi)\zeta,\zeta \rangle \leq \Lambda \vert \zeta \vert^2 \quad \zeta \in \R^2 \end{equation}      
is not satisfied \textit{both} from below ($\lambda=0$) and above ($\Lambda = + \infty)$. 
In this article we investigate what can be said about the regularity and possible singularities of $u$ when this set is arbitrary, hence the equation may be extremely degenerate. \\ \\ To fix ideas, one can first think of $G$ as the gradient of a $C^1$, strictly convex function \( F : \mathbb{R}^2 \to \mathbb{R} \). In this setting, equation \eqref{equation div} corresponds to the Euler-Lagrange equation for minimizers of the energy 
\begin{equation} 
\label{eq: Varformula}
\int_{B_1} F(\nabla u),             
\end{equation} for which the question of regularity is well understood in the uniformly elliptic setting, encoded in the inequality \eqref{eq: Elliptic inequality}, since the work of Morrey in dimension 2, and later De-Giorgi,Nash,Moser in arbitrary dimension. We investigate in this paper the case where the set $\mathcal{D}\cap\mathcal{S}$ of points $\xi\in \R^2$ for which \( D^2 F(\xi) \) has one eigenvalue equal to \( 0 \) and the other equal to \( +\infty \) is not empty, without imposing any conditions of any kind on its structure. As noticed in the introduction of \cite{DSS}, there is evidence suggesting that Lipschitz minimizers of \eqref{eq: Varformula} are in fact $C^1$ if $F$ is strictly convex, regardless of the structure of the set $ \mathcal{D}\cap \mathcal{S}$. 	This problem is solved in \cite{DSS} when $ \mathcal{D}\cap \mathcal{S}$ is empty or $ \mathcal{D}$ is finite, and in \cite{LacLam} when $ \mathcal{D}\cap \mathcal{S}$ is finite, but remains open for general strictly convex function $F$, which partially motivates the present work. \\

Interestingly, this conjecture can not be generalized to equation \eqref{equation div} with a general strictly monotone field $G$ which might not be a gradient. In that case, Lispchitz solutions might fail to be $C^1$, as shown in \cite[Theorem 1.5]{LacLam}. Before going further, we give the rigorous definitions of $\mathcal{D}$ and $\mathcal{S}$ for $G \colon \R^2 \to \R^2$ continuous, strictly monotone. They are given by
\begin{align}\label{eq:DS}
\begin{aligned}
\mathcal{D}(G) & =
\bigcap_{\lambda>0} \cl 
\left\lbrace \xi \in \mathbb{R}^2 : \liminf_{\vert  \zeta \vert \to 0} \frac{\langle G(\xi +\zeta) - G(\xi), \zeta \rangle}{\vert \zeta \vert^2} \leq\lambda \right\rbrace 
\\
\mathcal{S}(G) & =\bigcap_{\Lambda>0} \cl 
\left\lbrace \xi \in \mathbb{R}^2 : \liminf_{\vert \zeta \vert \to 0} \frac{\left\langle G( \xi +\zeta) - G(\xi), \zeta \right\rangle}{\left\vert G(\xi+ \zeta) - G(\xi) \right\vert^2} \leq
\frac 1\Lambda
\right\rbrace,
\end{aligned}
\end{align}
where $\cl(A)$ stands for the closure of a set $A \subset\R^2$, see the introduction of \cite{LacLam} for a more thorough discussion of these definitions.
Coming back to the example \cite[Theorem 1.5]{LacLam}, a particularly striking feature of it is that the gradient of the Lipschitz (but non-\( C^1 \)) solution satisfies \( \nabla u \in \mathcal{D} \cap \mathcal{S} \) almost everywhere. This naturally raises the question: what can be said about the regularity of \( \dist(\nabla u, \mathcal{D} \cap \mathcal{S}) \), where \( \operatorname{dist} \) denotes the standard distance function. Our main result shows that it is approximately continuous everywhere.

\begin{theorem}
\label{T: LooseTheorem}
Let $G:\R^2 \to \R^2$ strictly monotone, continuous and let $u : B_1 \to \R$ a Lipschitz solution of \eqref{equation div}.   Then, for any $x \in B_1$, either $x \mapsto \dist(\nabla u(x),\mathcal{D}\cap\mathcal{S})$ is continuous at $x$, or it holds 
\[ \frac{1}{\vert B_{\delta}\vert} \int_{B_{\delta}(x)} \dist(\nabla u(y),\mathcal{D}\cap\mathcal{S}) dy \to 0 \text{ as } \delta \to 0                       . \]
\end{theorem}

If $x\in B_1$ is such that $\dist(\nabla u(x),D\cap S) >0$, then thanks to Theorem \ref{T: LooseTheorem}  we know that in a neighborhood of $x$ the gradient $\nabla u$ takes values in an open set disjoint from $\mathcal{D}\cap \mathcal{S}$, hence $u$ is $C^1$ in that neighborhood due to \cite[Theorem 1.1]{DSS}. Hence, a direct Corollary of Theorem \ref{T: LooseTheorem} provides a description of the blow ups of $u$ : 

\begin{Corollary}
\label{th:Blowupversion}
Let $G : \R^2 \to \R^2$ continuous and strictly monotone. Then, for any Lipschitz solution of \eqref{equation div}, and $x_0 \in  B_1 $ : 
\begin{itemize}
\item Either there is a unique linear blow up limit, that is there exists $p\in \R^2$ such that : \[ \frac{u(x_0 + \delta x)-u(x_0)}{\delta} \underset{\delta \to 0}{\longrightarrow} \langle p,x \rangle \] 
\item Or any blowup limit : \[ v(x) := 
\lim \frac{u(x_0 + \delta x)-u(x_0)}{\delta}  \] along a subsequence $\delta \to 0$  satisfies $\nabla v \in \mathcal{D}\cap \mathcal{S}$ a.e.
\end{itemize}
\end{Corollary}
Under additional assumptions on $\mathcal{D}\cap \mathcal{S}$, one can hope to use this gradient localization property in order to further characterize the possible blowup limits. This idea is exploited in \cite{LamyTione} to prove partial $C^1$ regularity results.

From the point of view of the standard theory of elliptic PDE, it is reasonable to expect continuity of the function $\dist(\nabla u,\mathcal{D} \cup \mathcal{S})$, since the equation is uniformly elliptic if $\nabla u$ takes values in an open set contained in $\R^2\setminus (\mathcal{D}\cup \mathcal{S})$. In two dimensions, the estimates of \cite{DSS} even provide continuity of $\nabla u$ if it takes values in an open set contained in the possibly much larger $\R^2\setminus (\mathcal{D}\cap \mathcal{S})$, and continuity of $\dist(\nabla u,\mathcal{D}\cap \mathcal{S})$ is proved in \cite{LacLam} under the topological condition
\begin{equation}
\label{eq: Connexité condition}
 \R^2 \setminus \mathcal{N}^{\e} \left( \mathcal{D}\cap \mathcal{S} \right) \text{ is connected for small enough } \e>0, 
\end{equation} where $\mathcal{N}^\e(A)$ stands for the open $\e$-neighborhood of a set $A \subset \R^2$. We will discuss below several other previous works which obtained similar conclusions under some structural conditions on $\mathcal{D}\cap \mathcal{S}$. Our Theorem \ref{T: LooseTheorem} completely removes any condition on $ \mathcal{D}\cap \mathcal{S}$, but only provides a slightly weaker conclusion of approximate continuity.\\

Let us mention some simple examples for which Theorem \ref{T: LooseTheorem} provides new information. First, it applies to the counterexample to full $C^1$ regularity from \cite[Theorem 1.5]{LacLam}. Second, if $G$ is the gradient of a strictly convex radial function, $G=\nabla F$ where $F(x)=\phi(|x|)$ for some convex increasing $\phi\in C^1([0,\infty))$ such that $\phi'(0)=0$, and if one assumes in addition that $\phi$ is $C^2$ away from a finite number of radii $0<r_1<\cdots < r_N$, then, any Lipschitz minimizer $u$ of \[ \int_{B_1} F(\nabla u)       \] is either $C^1$ at $0$, or any blowup limit $v(x)=\lim \frac{u(\delta x)-u(0)}{\delta}$ along a sequence $\delta\to 0$ satisfies $|\nabla v|\in \lbrace r_j\rbrace$ a.e.

\subsection{Related works}
The closest result to the present work is contained in \cite{DSS}, and actually, our proof strongly relies on propositions already established in \cite{DSS}. In their article, De Silva and Savin studied a minimization problem of the type \eqref{eq: Varformula} in the framework of the obstacle problem, where $F$ is strictly convex and the minimizer is subject to the gradient constraint \( \nabla u \in N \), for some (closed) convex polygon \( N \subset \mathbb{R}^2 \), and the set $ \mathcal{D} \cap \mathcal{S} $ may contain the whole boundary $\partial N$ of that polygon. This allows minimizers to exhibit gradient discontinuities within \( B_1 \), even if \( F \) is smooth in the interior of \( N \). Their proof provides a detailed description of the behavior of such singularities. In particular, they are able to show that at any point of discontinuity $x_0$ of \( \nabla u \) and any $\e>0$, there exists $\delta>0$ such that $\nabla u(B_{\delta}(x_0)) \subset \mathcal{N}^{\e}(\partial N)$. Consequently, \( \operatorname{dist}(\nabla u, \mathcal{D} \cap \mathcal{S}) \) is continuous. 

In our case, we do not assume that \( \mathcal{D} \cap \mathcal{S} \) forms the boundary of a convex set. No structural assumption is required to establish regularity. It is worth noting that convexity assumptions regarding the "bad" set is crucial in their proof, and more generally in a lot of known results of this nature. 

For instance, let us mention \cite{SV}, where the authors studied the equation \( \operatorname{div}(\nabla F(\nabla u)) = f \) in two dimensions, with \( f \) belonging to a suitable \( L^p \)-space and \( F \) being a convex function that vanishes entirely within \( B_1 \). A typical example is \( F := (\lvert \cdot \rvert^2 - 1)_+ \). In this context, although \( \nabla F \) is not strictly monotone, the convexity of the set where \( F = 0 \) plays a key role in their main result. Specifically, they prove that \( H(\nabla u) \) is continuous for any continuous function \( H \) vanishing on \( B_1 \). This result was later extended in \cite{CF14a} to higher dimensions and to any function \( F \) vanishing on a \emph{strictly convex} set \( E \). Looking back at the example we gave above, we see that none of this theorems apply here. It appears challenging to avoid such structural assumptions, especially because the heart of those proofs relies on the existence of a non negative convex function $\Phi$ vanishing in the set where there is no ellipticity. The key is that $\Phi(\nabla u)$ is a subsolution of the equation because of the convexity of $\Phi$ and since $\Phi$ is $0$ in the bad set, one can adapt tools from the nondegenerate elliptic machinery
in order to obtain regularity of $\Phi(\nabla u)$. Such a function $\Phi$ does not exists in general in our framework, which is one of the main difficulties. \\
Some results, however, are available when convexity is relaxed. For example, in \cite{Lledos}, Lledos
obtained a similar conclusion in higher dimensions for the distance of
$\nabla u$ to the larger set $ \mathcal{D}\cup \mathcal{S}$ (instead of $\mathcal{D}\cap \mathcal{S}$ here), provided
it is contained within a two-dimensional plane, and any small
neighborhood of its connected components is simply connected. This condition needs to be compared to \eqref{eq: Connexité condition}, which is the only case where the continuity of $x \mapsto \dist(\nabla u(x),\mathcal{D}\cap \mathcal{S})$ is known. To understand why this condition is hard to remove in order to obtain continuity everywhere for the distance map, one can see \cite[Section 5]{CM}, or \cite[Remark 2.7]{LacLam}.

\vspace{1em}

Let us summarize the differences between the present work and existing results:
\begin{itemize}
\item We impose no conditions on the "bad" set \( \mathcal{D} \cap \mathcal{S} \) when $G$ is strictly monotone. In particular, we avoid any convexity-type assumptions or requirements for the neighborhoods of \( \mathcal{D} \cap \mathcal{S} \) to be simply connected.
    \item For \( G \) strictly monotone, we establish regularity for \( \operatorname{dist}(\nabla u, \mathcal{D} \cap \mathcal{S}) \), not just \( \operatorname{dist}(\nabla u, \mathcal{D}) \) (as in \cite{SV}, \cite{CF14a}), or \( \operatorname{dist}(\nabla u, \mathcal{D} \cup \mathcal{S}) \) (as in \cite{Lledos}).
    
    \item Our result is purely two-dimensional, we only consider zero right hand side, and we are not able to prove continuity everywhere for $x \mapsto \dist(\nabla u(x),\mathcal{D}\cap \mathcal{S})$.
\end{itemize}

\subsection{Strategy} We reduce Theorem \ref{T: LooseTheorem} to the case $x=0$ using the translation invariance of the equation. We separate two case whether or not $0$ is a Lebesgue point of $\nabla u$. If $0$ is a Lebesgue point, this means that $u$ is very close to a linear function with slope $p$ at some scale $\rho$ (see Proposition \ref{Contprop}). Using a result of \cite{DSS}, this implies that at scale $\rho/2$, the image $\nabla u(B_{\rho/2}) $ does not touch any ball $B_{\eta}(q)$ included in an elliptic region provided $p \notin B_{2\eta}(q)$. The issue is to find such a ball; we show that it always exists because of the strict monotonicity of $G$, see Lemma \ref{structure}.  At this point, we can use the localisation Theorem \ref{p:apriori} proved in \cite{LacLam} to infer, for $r>0$ the existence of $\delta>0$ such that $ \nabla u(B_{\delta \rho/2}) \subset B_{r}(p)$ if $p$ is far from the bad set, or $ \nabla u(B_{\delta \rho/2}) \subset \mathcal{N}^r(\mathcal{D}\cap \mathcal{S}) $ if $p$ is in the bad set. In any case, the oscillations of the distance to $\mathcal{D}\cap \mathcal{S}$ is small, and this provides continuity. \\
If $0$ is not a Lebesgue point, the argument is based on \cite[Lemma 5]{SV}, which relates superlevel sets of a given $H^1$ function in a two-dimensional disk to its Dirichlet energy in annuli. We combine this property with an energy estimate away from $\mathcal{D}$, due to \cite{DSS}, and with the maximum principle, to estimate the set where $\nabla u$ lies far from $ \mathcal{D}$, and eventually from $\mathcal{D}\cap \mathcal{S}$ thanks to a duality used already in \cite{LacLam}.

\subsection{Plan}
In section 2, we recall the approximation procedure and the a priori estimate of \cite{LacLam}. In section 3, we prove a topological result on the sets $\mathcal{D}$ and $\mathcal{S}$. Section 4 is devoted to the proof of the localisation result we need to prove the first part of Theorem \ref{T: LooseTheorem}. In section 5, we prove the continuity of $x \mapsto \dist(\nabla u(x),\mathcal{D}\cap \mathcal{S})$ at any Lebesgue point $x$ of $\nabla u$. In section 6, we prove the second part of Theorem \ref{T: LooseTheorem} concerning blow ups at non Lebesgue point of $\nabla u$, it can be read without section 3,4,5. Finally, we give in Appendix the proof of a Lemma stated in \cite{DSS} and the proof of the a priori estimate of \cite{LacLam}.

\section{Preliminary results}

\subsection{Basic Definitions and notations}

Those are the features of a strictly monotone, continuous field $G$ subsequently used in our analysis : 

\begin{itemize}
\item
the \emph{modulus of monotony} $\omega_G\colon (0,\infty)\to (0,\infty)$, given by
\begin{align*}
\omega_G(t)=\inf_{|\xi-\zeta| > t} \langle G(\xi)-G(\zeta),\xi-\zeta\rangle,
\end{align*}
\item the open sets $O_\lambda(G)$, $V_\Lambda(G)$ given, as in \cite{DSS},
 by
\begin{align}
O_\lambda(G)
&
=
\interior
\left\lbrace
\xi\in\R^2\colon
\liminf_{\vert \zeta \vert \to 0} \frac{\langle G(\xi +\zeta)-G(\xi),\zeta \rangle}{\vert \zeta \vert^2} \geq \lambda
\right\rbrace
\, ,
\label{eq:Olambda}
\\
V_\Lambda(G)
&
=
\interior
\left\lbrace
\xi\in\R^2\colon
\liminf_{\vert \zeta \vert \to 0} \frac{\langle G(\xi+\zeta)-G(\xi),\zeta \rangle}{\vert G(\xi+\zeta)-G(\xi) \vert^2} \geq \frac{1}{\Lambda}
\right\rbrace
\, ,
\label{eq:VLambda}
\end{align}
for $\lambda,\Lambda>0$, where $\interior A$ denotes the topological interior of $A\subset\R^2$. The relevance of this sets lies on the equalities : \[ \mathcal{D}= \R^2 \setminus \bigcup_{\lambda>0} O_{\lambda} \quad \quad \mathcal{S} = \R^2 \setminus\bigcup_{\Lambda>0}V_{\Lambda}                \]
\end{itemize}

In the approximation, we will work with \emph{strongly monotone} vector field $G$, that is, there exists $C>0$ such that
\begin{align}\label{eq:strongmonot}
C \langle G(\xi)-G(\zeta),\xi-\zeta\rangle \geq |\xi-\zeta|^2 + |G(\xi)-G(\zeta)|^2
\end{align}

\subsection{A priori estimates for smooth solutions}

The following Theorem \ref{p:apriori} is referred to as the "localisation Theorem" for smooth solutions. Heuristically, it provides a scale $\delta$ such that if one knows that the image of $B_1$ by a smooth solution $u$ does not touch a small ball in a region of type $O_{\lambda}$ or $V_{\Lambda}$ (see \eqref{eq:Olambda},\eqref{eq:VLambda}), then, at scale $\delta$, $\nabla u$ has to localise in an elliptic region in the connected component containing $O_{\lambda}$ (or $V_{\Lambda}$), or $\nabla u(B_{\delta}) $ is outside this connected component. This Theorem was obtained in \cite{LacLam} based on the work of \cite{DSS}. See \cite[Proposition 2.1 and Remark 2.7]{LacLam}. Since the statement is slightly different in this context, we give a proof in the appendix. \\ 

Until the end, the \textbf{open} $r$-neighborhood of a set $A$ is written $\mathcal{N}^r(A)$. 
\begin{theorem}\label{p:apriori}
Let $G \colon \mathbb{R}^2 \to \mathbb{R}^2 $ smooth and strongly monotone.
Assume that there exist $\lambda,\Lambda, M > 0$ and an open set $U$ such that
\begin{align*}
\overline {B_{2M}} \subset  V_{\Lambda}(G) \cup O_{\lambda}(G) \cup U.
\end{align*} 
Let $u$ any smooth solution of $\dv(G(\nabla u))=0$ in $B_1$
with $|\nabla u|\leq M$. Assume that for some connected component $\mathcal{C}$ of $ \overline{B_{2M}} \setminus U $, there exists a ball $B_{\rho}(q) \subset \mathcal{C} $  such that 
\[ \nabla u(B_1) \cap B_{\rho}(q) = \emptyset        \]
Then, for any Lebesgue number $\eta\in (0,\rho)$ of the above open covering ( 
any ball $B_\eta(\xi)$ with $ \vert \xi \vert \leq 2M$ must be contained in $V_{\Lambda}(G)$, $O_{\lambda}(G)$ or $ U $ )  there exists $\delta > 0$ such that either :
\begin{align*}
\nabla u(B_{\delta}) \subset B_{\eta}(p), \text{ for some } p \in \mathcal{C}
\end{align*}
or
\begin{align*}
\nabla u (B_{\delta}) \subset \R^2 \setminus \mathcal{C}, 
\end{align*}
 where
$\delta>0$ depends on
\begin{itemize}
\item The Lebesgue number $\eta$ ;
\item  the gradient bound $M$ and 
the ellipticity constants $\lambda,\Lambda$;
\item the integrals $\int_{B_1}|\nabla u|^2\, dx $ and $\int_{B_1} |G(\nabla u)|^2\, dx$;
\item the modulus of monotony $\omega_G$ via any $c>0$ such that
$\omega_G(t)/t\geq c$ for all
 $t \in [\eta/4,M+\eta]$.
\end{itemize}
\end{theorem}

In order to apply Theorem \ref{p:apriori}, we need to find a ball $B_{\rho}(q)$ satisfying the condition $\nabla u (B_1) \cap B_{\rho}(q) = \emptyset$. The following Lemma is precisely used to find such a ball when we know that $u$ is close to an affine function. It is the claim 8.2 of \cite[Proposition 6.2]{DSS}, we sketch its proof in the appendix.
\begin{lemma}
\label{Lemme inter vide}
Let $G$ smooth strongly monotone, and $u$ a smooth solution of $\dv(G(\nabla u))=0$ with Lipschitz bound $M$. 
Let $p,q \in \R^2 $ and $\rho >0 $ such that $B_{\rho}(q) \subset O_{\lambda}\cap V_{\Lambda} $ and $p \notin B_{\rho}(q)$. There exists $\e := \e(\rho,\lambda,\Lambda) $ such that for any smooth solution $u$ satisfying
\[ \vert u(x) - l_p(x) \vert \leq \e \text{ for all } x \in B_1            \] for some affine map $l_p$ with $\nabla l_p = p $, 
then, $\nabla u(B_{1/2}) \cap B_{\rho/2}(q)=\emptyset$.
\end{lemma}

One can check the proof in the Appendix for the explicit dependence on $\rho,\lambda,\Lambda$.

\subsection{The approximation}

We recall in this section the procedure to approximate the solution of $\dv(G(\nabla u))=0$ by smooth solutions $\nabla u_{\e}$. All the proofs can be found in \cite{LacLam}. The reader may skip this section in first reading. \\
The following Lemma is the first step in the construction of the approximation. It is a modification of a given field $G$ at infinity, which is needed to control the Lipschitz norm of the approximation $u_{\e}$ independantly of $\e$.

\begin{lemma}\label{l:modif}
Let $G \colon \mathbb{R}^2 \to \mathbb{R}^2 $ a continuous strictly monotone vector field, 
and $M>0$.
%such that $\mathcal{D}(G) \cap \mathcal{S}(G) \cap \overline{B_M}$ is finite for some $M>0$.
Then there exists $\widetilde G\colon\R^2\to\R^2$ a continuous strictly monotone vector field  equal to $G$ in $\overline B_M$ and smooth outside $B_{4M}$, such that
\begin{align*}
\mathcal{D}(\widetilde G)\cap\mathcal S(\widetilde G)\subset \mathcal{D}(G) \cap \mathcal{S}(G) \cap \overline{B_M},
\end{align*}
and
\begin{align*}
&
c\leq \nabla^s \widetilde G(\xi) \leq |\nabla \widetilde G(\xi)|\leq 4c
\quad \forall \xi \in \R^2\setminus B_{4M},
\\
&
|\widetilde G(\xi)|\leq L(1+|\xi|)\quad\forall \xi\in\R^2,
\end{align*}
for some constants $L,c>0$ depending on $M$ and $\|G\|_{L^\infty(B_{4M})}$.
\end{lemma}

Secondly, we construct the approximation by smoothing the field $\tilde{G}$ obtained previously. The standard theory of elliptic PDE's ensures that the solutions of $\dv(G_{\e}(\nabla u_{\e}))=0$ are smooth.

\begin{lemma}
\label{l:approxG}
Let $G \colon \mathbb{R}^2 \to \mathbb{R}^2 $ a continuous strictly monotone vector field. 
Assume that there exist $M,L\geq 1$, $c >0$ such that $G$ is smooth in $\R^2\setminus B_{4M}$
and
\begin{align*}
&
c\leq \nabla^s  G(\xi) \leq |\nabla  G(\xi)|\leq 4c
\quad \forall \xi \in \R^2\setminus B_{4M},
\\
&
| G(\xi)|\leq L(1+|\xi|)\quad\forall \xi\in\R^2.
\end{align*}
Then there exists a sequence $G_\e$ of smooth and 
strongly monotone \eqref{eq:strongmonot} vector fields such that
$G_\e\to G$ locally uniformly as $\e\to 0$, and
\begin{align*}
&
\nabla^s G_\e(\xi) \geq c
\qquad \forall \xi \in \R^2\setminus B_{5M},
\\
&
| G_\e(\xi)|\leq 2L(1+|\xi|)\qquad\forall \xi\in\R^2\, 
,
\\
&
 \omega_{G_{\e}} \geq \omega_G, 
 \\
&
B_{2\e}(\xi)\subset O_\lambda(G)\;\Rightarrow\; \xi\in O_\lambda(G_\e)\, ,
%\quad\forall \lambda>0,
\\
& 
B_{2\e}(\xi)\subset V_{\Lambda}(G)\;\Rightarrow\; \xi\in V_{\Lambda+\e}(G_\e)\, ,
%\quad\forall \Lambda >0.
\end{align*}
for all $\e\in (0,1)$.
\end{lemma}

Finally, by standard methods, one obtain that $u_{\e}$ converges to the solution of the initial problem : 

\begin{lemma}\label{l:approxu}
Let $G,G_\e \colon \R^2 \to \R^2 $ be as in Lemma~\ref{l:approxG}, 
and $u$
 a solution of $\dv G(\nabla u)=0$ in $B_1$ with $|\nabla u|\leq M$. 
 For $\e\in (0,1)$, let $u_{\e}$ be the unique smooth solution of the boudary value problem 
\begin{align*}
\dv G_{\e}(\nabla u_{\e} )=0 
\quad \text{in } B_1,
\qquad
u_{\varepsilon}=u  
\quad \text{in } \partial B_1.
\end{align*} 
Then we have
\begin{align*}
\sup_{\e\in (0,1)}\|\nabla u_\e\|_{L^\infty(K)} <\infty\qquad\text{ for all compact }K\subset B_1,
\end{align*}
and $u_{\e} \to u $ locally uniformly in $B_1$, and strongly in $W^{1,2}(B_1)$.
\end{lemma}

\section{Topological structure of $\mathcal{D}$ and $\mathcal{S}$ }

In this section, we gather two results on the sets $\mathcal{D}, \mathcal{S}$. Those are used to apply Lemma \ref{Lemme inter vide}.

\begin{lemma}
\label{structure}
Let $G : \R^2 \to \R^2 $ strictly monotone and continuous. Then, $ \mathcal{D}\setminus \mathcal{S} $ has empty interior.  
\end{lemma} 
 
\begin{proof}
Assume that $B_r(p) \subset \mathcal{D}\setminus \mathcal{S} $ for some $r>0$. Since \[\mathcal{S} = \R^2 \setminus \bigcup_{\Lambda>0} V_{\Lambda}(G), \] there exists $\Lambda_0 > 0$ such that $B_r(p) \subset \mathcal{D} \cap V_{\Lambda_0}(G)$. It follows that $G$ is $\Lambda_0$ Lipschitz in $B_r(p)$ (see \cite[Lemma A.3]{LacLam}), and we may apply the area formula, see for instance \cite[Theorem 1.3]{CafCab} (the jacobian is non negative by strict monotonicity) to obtain \[ \vert G(B_r(p)) \vert \leq \int_{B_r(p)} \det \nabla G         . \] Using the inclusion $B_r(p) \subset \mathcal{D} $, we deduce that $\nabla G $ has an eingenvalue equal to $0$ at each point where $G$ is differentiable, so $\det \nabla G = 0 $ a.e in $B_r(p)$ since $G$ is also Lipschitz in $B_r(p)$. From this, we infer $ \vert G(B_r(p)) \vert = 0 $. However, since $G$ is injective (because strictly monotone) and continuous, the invariance of the domain Theorem (see \cite{Brouwer}) ensures that $G$ is an open map. This is a contradiction.
\end{proof}

Next, we prove the counter part on $\mathcal{S}\setminus \mathcal{D}$.

\begin{lemma}
\label{l:structureSD}
Let $G : \R^2 \to \R^2 $ strictly monotone and continuous. Then, $ \mathcal{S}\setminus \mathcal{D}$ has empty interior.  
\end{lemma}

\begin{proof}
We recall that we're interested in Lipschitz solutions of \eqref{equation div}, therefore it is sufficient for our purpose to prove that $(\mathcal{S}\setminus \mathcal{D} ) \cap B_{M}$ has empty interior for all $M>0$. This is why, in the proof, we are free to modify $G$ as we want outside an arbitrary large compact subset of $\R^2$. \\
We write $i$ denotes the counter-clockwise rotation of angle $\pi/2$. Following \cite[Proposition 4.1]{LacLam} one may modify $G$ outside a sufficiently large ball $\overline{B_{M}} $ to ensure that $G$ is an homeomorphism (see \cite[Proposition 4.1 point 1,2]{LacLam}). Therefore, the field \[G^* : \xi \in \R^2 \mapsto iG^{-1}(-i\xi)   \] is well defined, continuous and strictly monotone (see \cite[Proposition 4.1 point 3]{LacLam}).
Using \cite[Proposition 4.1 point 4]{LacLam}  we have \[  iG(\mathcal{S}(G))= \mathcal{D}(G^*) \] and similarly : \[ iG(\mathcal{D}(G)) =  \mathcal{S}(G^*). \] Since $G^*$ is strictly monotone, we apply Lemma \ref{structure} and we deduce that the set $  iG(\mathcal{S}(G))\setminus iG(\mathcal{D}(G)) $ has empty interior. Since $G$ is an homeomorphism, we get that $\mathcal{S}(G) \setminus \mathcal{D}(G)$ has empty interior.
\end{proof}

\begin{Corollary}
\label{Structurebis} Let $G:\R^2 \to \R^2$ continuous, strictly monotone.
For any non empty open set $U \subset \R^2 \setminus \mathcal{D}\cap \mathcal{S}$, there exists $\lambda,\Lambda,\eta > 0 $ and $q \in U$ such that $B_{\eta}(q) \subset U \cap O_{\lambda} \cap V_{\Lambda} $ 
\end{Corollary}

\begin{proof}
Fix $U$ as in the statement. First, we prove that there exists $\lambda>0$ such that $U\cap O_{\lambda} \neq \emptyset $. Indeed, assume $U\cap O_{\lambda}= \emptyset $ for any $\lambda >0$, then, $U\subset \mathcal{D}$. However, since $U \subset \R^2 \setminus (\mathcal{D}\cap \mathcal{S})$ we can not have $U\cap \mathcal{S} \neq \emptyset$ otherwise there would be a point belonging in $U$ and $ \mathcal{D}\cap \mathcal{S}$. Hence, choosing a ball $B_{r}$ in the interior of $U$, we find $B_{r} \subset \mathcal{D}\setminus \mathcal{S}$, which is a contradiction with Lemma \ref{structure}. To conclude, fix such a $\lambda>0$ and assume $U\cap O_{\lambda}\cap V_{\Lambda} = \emptyset $ for any $\Lambda>0$. Then, $ U\cap O_{\lambda} \subset \mathcal{S} $, and since $ U\cap O_{\lambda}$ is open non empty, we deduce that $ \mathcal{S}\setminus \mathcal{D}$ has non-empty interior, which contradicts Lemma \ref{l:structureSD}. Thus, there exists $\Lambda>0$ such that $U\cap O_{\lambda}\cap V_{\Lambda} \neq \emptyset $, since those three sets are open, the proof is complete.
\end{proof}

\section{Localisation result}
In this section, we recall the flatness results of \cite{DSS}. We rely on \cite[Proposition 6.2 and 6.3]{DSS} to obtain a new localisation property using Theorem \ref{p:apriori}.

\begin{proposition}
\label{p:Flatness}
Let $G$ smooth, strongly monotone, and $u$ a smooth solution of $\dv(G(\nabla u))=0$. Assume that there exists $0< \lambda< \Lambda < \infty$ , $ \rho,M > 0 $, $q \in \R^2$ and an open set $U$ such that : 
\begin{align}
\label{CoveringB2M}
& \overline{B_{2M}} \subset O_{\lambda}(G) \cup V_{\Lambda}(G) \cup U, \\
\label{Boule dans l'intersection}
& B_{\rho}(q) \subset \left(O_{\lambda}(G) \cap V_{\Lambda}(G) \right) \setminus \bar U.
\end{align}
Then, for any $r>0$, there exists $\e,\delta > 0$ depending on : 

\begin{itemize}
\item A Lebesgue number $\eta \in (0, \min(\rho/4,r/2)) $ of the covering \eqref{CoveringB2M} ( any ball $B_{\eta}(\xi)$ is contained either in $O_{\lambda},V_{\Lambda}$ or $U$ )
\item The quantities $\lambda,\Lambda$, $ \| G(\nabla u ) \|_{L^2(B_1)} $, $ \| \nabla u \|_{L^2(B_1)} $ \\
\item The monotony modulus $\omega_G$ via any $c>0$ such that $c \geq \inf_{(\eta/4, M+\eta)} \frac{\omega_G(t)}{t}$
\end{itemize}
such that if 
\begin{equation}
\label{u flat}
\vert u(x)-l_{p_0}(x) \vert \leq \e, \quad \forall x \in B_1 
\end{equation}
for some affine map $l_{p_0}$ with $\nabla l_{p_0} = p_0 \notin B_{\rho}(q)$, then either 
\begin{equation}
\label{conclusion 1}
\nabla u(B_{\delta}) \subset B_r(p_0)
\end{equation}
or 
\begin{equation}
\label{conclusion 2}
\nabla u(B_{\delta}) \subset V
\end{equation}
where $V$ is the complement of the connected component of $ \overline{B_{2M}} \setminus U$ which contains $q$.  \\
\end{proposition}

\begin{proof}
The proof is the same as in \cite{DSS}, but with a different localisation Theorem. We denote $\mathcal{C}$ the connected component of $\overline{B_{2M}} \setminus U $ which contains $q$.  Applying  Lemma \ref{Lemme inter vide} we can choose $\e:= \e(\rho,\lambda,\Lambda) $ small enough so that \eqref{u flat} implies $\nabla u(B_{1/2}) \cap B_{\rho/2}(q) = \emptyset$. We fix a Lebesgue number $\eta$ of the covering \eqref{CoveringB2M} with $0<\eta \leq \min(\rho/4,r/2)$, and we apply Theorem \ref{p:apriori} to obtain $\delta > 0 $ such that one of the two alternatives arises : 
\begin{align}
\label{loc}
\nabla u(B_{\delta/2}) \subset B_{\eta}(p) \text{ for some } p\in \mathcal{C} \text{  or  } \nabla u (B_{\delta/2}) \subset \overline{B_{2M}} \setminus \mathcal{C}. 
\end{align} \\ We distinguish the two cases by looking at the position of $p_0$.  \\

\begin{itemize}
\item Case 1 : $B_r(p_0) \subset \mathcal{C}. $ 
\end{itemize} 

Observe that in the region where $ \vert x \vert^2 \geq 4\e$ we have  $\frac{1}{2} \vert x \vert^2 + u(x)-p_0 \cdot x \geq \e$ because of the flatness assumption \eqref{u flat}. Without loss of generality, we can assume $u(0)=0$, and therefore, the map $x \mapsto \frac{1}{2} \vert x \vert^2 + u(x)-p_0 \cdot x$ has a minimum $x_0 \in B_{2\sqrt{\e}}(0)$. At this point, we must have $ \nabla u(x_0)= p_0-x_0$, and we deduce 
\begin{equation}
\label{eq 0}
\nabla u(B_{2\sqrt{\e}}(0)) \cap B_{2 \sqrt{\e}}(p_0) \neq \emptyset. 
\end{equation} 
Eventually lowering $\e$, we also have $2\sqrt{\e}\leq \frac{\delta}{2} < r$, so because of the inclusion $B_r(p_0) \subset \mathcal{C}$, we can not have the second case in \eqref{loc}. Hence, $\nabla u(B_{\delta/2}) \subset B_{\eta}(p)$ for some $p \in \mathcal{C}$. Also, from \eqref{loc} and \eqref{eq 0} we get $ \vert p-p_0 \vert \leq \eta + 2\sqrt{\e} \leq 2\eta  $, proving that $\nabla u(B_{\delta/2}) \subset B_{2r}(p_0)$.

\begin{itemize}
\item Case 2 : $B_r(p_0) \subset \overline{B_{2M}} \setminus \mathcal{C}$.
\end{itemize}
We argue as above, but this time \eqref{eq 0} gives the second inclusion in \eqref{loc}.

\end{proof}

\section{Continuity at Lebesgue points of $\nabla u$}
In this section, we prove the first part of our main result, namely :

\begin{theorem}
\label{Th inter}
Let $G:\R^2 \to \R^2 $ continuous, strictly monotone. Let $u: B_1 \to \R$ a $M$-Lipschitz solution of \eqref{equation div}. Assume $0$ a Lebesgue point of $\nabla u$. Then, $x \mapsto \dist(\nabla u(x),\mathcal{D}\cap \mathcal{S}) $ is continuous at $0$.
\end{theorem}

To prove this Theorem, we we combine the approximation procedure define in section 2.3 and the Localisation results of section 4. To do so, we need to check that $u$ is $\e$-close to an affine map at some scale. The following Lemma is precisely used in this purpose. The proof is standard and can be found in \cite[Theorem 6.5]{EvGa} 

\begin{lemma}
\label{Contprop}
Let $u :B_1 \to \R$ be a $M$-Lipschitz map, and assume that $0$ is a Lebesgue point of $\nabla u$. Then, for any $\e>0$, there exists $\rho > 0$ such that 
\[ \left \vert u(x)-\langle p,x \rangle - u(0) \right\vert \leq \e\rho \text{ for all } x \in B_{\rho}(0)           \] where $  \lim_{\delta \to 0} \frac{1}{\vert B_{\delta}\vert}\int_{B_{\delta}} \vert \nabla u(y)-p\vert dy = 0 $.
\end{lemma}

We are ready to prove Theorem \ref{Th inter}.

\begin{proof}[Proof of Theorem \ref{Th inter}]

First, we assume that $\mathcal{D}\cap \mathcal{S} \cap \overline{B_M} \neq \emptyset$, otherwise, $u \in C^1$ by \cite[Main Theorem]{LacLam} and there is nothing to prove.\\

Let $p_0$ such that \[ \lim_{\delta \to 0} \frac{1}{\vert B_{\delta}\vert}\int_{B_{\delta}}  \vert \nabla u(y)-p_0 \vert dy = 0.\] 

The idea is to apply Propositions \ref{p:Flatness} relatively to the position of the vector $\nabla u(0)$. By Lemma \ref{l:modif}, we do not change the equation by replacing $G$ with $\tilde{G}$ and therefore, we can assume that \[ \mathcal{D}(G)\cap \mathcal{S}(G) \subset \overline{B}_M.         \] Thanks to Lemma \ref{l:approxG} there exists a sequence of smooth, strongly monotone vector field $G_{\e}$ such that $ \omega_{G_\e} \geq \omega_G$ and for all $0 < \lambda <\Lambda< + \infty $  : 

\begin{align}
\label{Oapprox}
B_{2\e}(\xi) \subset O_{\lambda}(G) & \Rightarrow \xi \in O_{\lambda}(G_{\e})
\end{align}
and 
\begin{align}
\label{Vapprox}
B_{2\e}(\xi) \subset V_{\Lambda}(G) & \Rightarrow V_{\Lambda + \e}(G_{\e}).
\end{align}

We apply Lemma \ref{l:approxu} to obtain a sequence $u_{\e}$ of smooth solutions of the equation $\dv(G_{\e}(\nabla u_{\e}))=0$ which satisfies $ \vert \nabla u_{\e} \vert \leq \tilde{M}$ in $B_{1/2}$ for some $\tilde{M}\geq M$ and converges locally uniformly to $u$ in $B_1$. Rescaling to $B_1$, we assume $\vert \nabla u_{\e} \vert \leq \tilde{M}$ in $B_1$. We fix a sequence $\e(m) \to 0$ and write $G_m= G_{\e(m)}$, $u_m=u_{\e(m)}$. For $r>0$, we check that there exists $\delta > 0 $ independant of $m$ such that \[ \diam(\dist(\nabla u_m(B_{\delta})) \leq r.   \]

We distinguish weather or not $p_0$ belongs to $\mathcal{D}\cap \mathcal{S}$. \\

\noindent\textbf{Case 1.} $p_0\notin \mathcal{D}\cap\mathcal{S}. $  \\  
Since $p_0 \notin \mathcal{D}\cap \mathcal{S}$, eventually lowering $r$, we assume 
\begin{align}
\label{eq7}
B_r(p_0) \cap \mathcal{N}^r(\mathcal{D}(G)\cap \mathcal{S}(G))=\emptyset. 
\end{align}

Let $\mathcal{C}$ be the connected component of $\overline{B}_{2\tilde{M}} \setminus \mathcal{N}^r(\mathcal{D}\cap\mathcal{S}) $ which contains $B_r(p_0)$. By Corollary \ref{Structurebis} there exists $0<\lambda_0<\Lambda_0< \infty $ and a ball 
\begin{equation}
\label{Boule loin de p_0}
B_{\rho}(q) \subset O_{\lambda_0}(G)\cap V_{\Lambda_0}(G) \cap \mathcal{C}.
\end{equation}
Replacing $\rho$ with $\rho/2$ and eventually changing the localisation of the point $q$, we may assume that $p_0 \notin B_{\rho}(q)$. Now, as $\mathcal{D}$ and $\mathcal{S}$ are the complement of $\bigcup_{\lambda>0}O_{\lambda}(G)$ and $ \bigcup_{\Lambda>0} V_{\Lambda}(G)$, there exists $\infty>\Lambda_1>\lambda_1>0$ such that
\begin{equation}
\label{Covering0}
\overline{B_{2 \tilde{M}}} \setminus \mathcal{N}^r(\mathcal{D}\cap\mathcal{S}) \subset O_{\lambda_1}(G) \cup V_{\Lambda_1}(G). 
\end{equation} 

Set $\lambda = \min(\lambda_0,\lambda_1)$ and $\Lambda = \max(\Lambda_0,\Lambda_1)$. By monotonicity of those sets, \eqref{Boule loin de p_0} and \eqref{Covering0}, we have : 

\begin{align*}
& \overline{B_{2\tilde{M}}} \subset O_{\lambda}(G)\cup V_{\Lambda}(G)\cup \mathcal{N}^r(\mathcal{D}\cap \mathcal{S}), \\
& B_{\rho}(q) \subset O_{\lambda}(G)\cap V_{\Lambda}(G) \setminus \overline{(\mathcal{N}^r(\mathcal{D}\cap\mathcal{S}))}, \\
& p_0 \notin B_{\rho}(q).
\end{align*}

We write $\mathcal{C}_m$ the connected component of \[\overline{B}_{2\tilde{M}} \setminus (\mathcal{N}^r(\mathcal{D}\cap \left(\mathcal{S})) \cap (O_{\lambda}(G_m)\cup V_{2\Lambda}(G_m) \right) \] which contains $B_{\rho}(q)$.

 By \eqref{Oapprox} and \eqref{Vapprox}, we deduce that for all $m$ large enough :
\begin{align}
\label{eq8}
& \overline{B_{2\tilde{M}}} \subset V_{2\Lambda}(G_m) \cup O_{\lambda}(G_m)\cup \mathcal{N}^r(\mathcal{D}\cap \mathcal{S}),        \\
\label{eq9}
& B_{\rho}(q) \subset O_{\lambda}(G_m) \cap V_{2\Lambda}(G_m) \setminus \overline{(\mathcal{N}^r(\mathcal{D}\cap\mathcal{S}))}, \\
\label{eq10}
& p_0 \notin B_{\rho}(q). 
\end{align} 

Since $G_m$ is smooth, strongly monotone and $u_m$ is a smooth solution of the associated equation, and because of \eqref{eq8},\eqref{eq9},\eqref{eq10}, we can apply the rescaled version of Proposition \ref{p:Flatness} : there exists $\e > 0$ such that if for some $t>0$ we have
\[ \sup_{x \in B_t}            \vert u_m(x) - l_{p_0}(x) \vert \leq \e t \quad l_{p_0} \text{ affine with } \nabla l_{p_0}=p_0 \] 
then, $\nabla u_m(B_{t\delta/2}) \subset B_{r}(p_0) $ (the first case of Proposition \ref{p:Flatness} arises because $B_r(p_0) \subset \mathcal{C}_m$). Here, we recall that $\e,\delta$ depends on : 
\begin{itemize}
\item A Lebesgue number $\eta \in (0, \min(\rho/4,r/2)) $ of the covering \eqref{eq8} (any ball $B_{\eta}(\xi)$ is contained either in $O_{\lambda}(G_m),V_{2\Lambda}(G_m)$ or $\mathcal{N}^r(\mathcal{D}\cap\mathcal{S}))$).
\item The quantities $\lambda,\Lambda$ ,$ \| G_m(\nabla u_m ) \|_{L^2(B_1)} $ and $ \| \nabla u_m \|_{L^2(B_1)} $. \\
\item The monotony modulus $\omega_{G_m}$ via any $c>0$ such that $c \geq \inf_{(\eta/4, \tilde{M}+\eta)} \frac{\omega_{G_m}(t)}{t}$.
\end{itemize}

By Lemma \ref{l:approxG}, we have $\omega_{G_m}\geq \omega_G$ and any Lebesgue number of the covering \[ \overline{B_{2\tilde{M}}} \subset O_{\lambda}(G)\cup V_{\Lambda}(G)\cup \mathcal{N}^r(\mathcal{D}\cap \mathcal{S})          \] is still a Lebesgue number of the covering \eqref{eq8}. Also, by Lemma \ref{l:approxu}, we can  bound $ \| G_m(\nabla u_m ) \|_{L^2(B_1)} $ and $ \| \nabla u_m \|_{L^2(B_1)} $  uniformly in $m$. This means that we can choose $\e$ and $\delta $ uniform with respect to $m$. Fix such an $\e>0$ and let us check that $u_m$ is $\e$ close to an affine map for $m$ large enough. \\

Since $0$ is a Lebesgue point of $\nabla u$, by Lemma \ref{Contprop}, for any $\e> 0$ there exists $t := t(\e) $ for which \[ \sup_{x\in B_{t}} \left \vert u(x)-\langle p_0,x\rangle - u(0)\right \vert \leq \frac{\e}{2} t.              \]  Since $u_m$ converges locally uniformly toward $u$ (see Lemma \ref{l:approxu}), we obtain for all $m$ large enough :

\begin{align}
\label{equ_m}
\sup_{x\in B_{t}} \left \vert u_m(x)-\langle p_0,x \rangle- u_m(0) \right\vert \leq \e t.
\end{align}
Applying Proposition \ref{p:Flatness}  to \eqref{equ_m} gives the inclusion \[\nabla u_m(B_{t \delta/2}) \subset B_{r}(p_0) \] with $\delta$ and $t$ independant of $m$. Since the distance function is $1$ Lipschitz, we deduce that \[ \diam(\dist(\nabla u_m(B_{t\delta/2}),\mathcal{D}\cap \mathcal{S})) \leq r \] for all $m$ large enough. Letting $m$ going to infinity gives the result. \\

\noindent\textbf{Case 2.} $p_0 \in \mathcal{D}\cap \mathcal{S}$ \\

We first introduce the quantity :
\[K  = \text{ Nbr of connected component of } \overline{B_{2\tilde{M}}} \setminus \overline{\mathcal{N}^r(\mathcal{D}\cap \mathcal{S})}\]and we check that $K$ is finite for any $r>0$. Indeed, writing $\mathcal{C}_i$ the different connected component, we have \[ \mathcal{N}^{r/2}(\mathcal{C}_i)\cap \mathcal{N}^{r/2}(\mathcal{C}_j) = \emptyset \quad  \forall i\neq j \in \lbrace 1,,,K \rbrace. \] Since \[ \bigcup_{i=1}^{K} \mathcal{N}^{r/2}(\mathcal{C}_i) \subset \mathcal{N}^{r/2}(\overline{B_{2\tilde{M}}})              \] we deduce \[ \pi K \frac{r^2}{4} \leq K \min_{i \in \lbrace 1,,, K\rbrace} \lbrace \vert                  \mathcal{N}^{r/2}(\mathcal{C}_i) \vert \rbrace \leq \sum_{i=1}^{K} \vert \mathcal{N}^{r/2}(\mathcal{C}_i) \vert) \leq \pi (2\tilde{M}+r/2)^2  \] therefore \[ K \leq \frac{4(2\tilde{M}+r/2)^2}{r^2} < \infty \quad \forall r>0.                  \] \\

Now, we fix $r>0$ and denote $\mathcal{C}$ a connected component of \[ B_{2 \tilde{M}} \setminus \overline{\mathcal{N}^r(\mathcal{D}\cap \mathcal{S})}. \] In this connected component, we apply Corollary \ref{Structurebis} to obtain a ball $B_{\rho}(q)$ and numbers $0<\lambda<\Lambda< \infty $ such that \[ B_{\rho}(q) \subset O_{\lambda}(G)\cap V_{\Lambda}(G)\cap \mathcal{C}               \] with the $p_0 \notin B_{\rho}(q)$. Similarly to what we did above, we write $C_m$ the connected component of \[ \overline{B}_{2\tilde{M}} \setminus (\overline{\mathcal{N}^r(\mathcal{D}\cap \mathcal{S}}) \cap \left(O_{\lambda}(G_m)\cup V_{2\Lambda}(G_m) \right)  \] which contains $B_{\rho}(q)$.  Proposition \ref{p:Flatness} gives $\e$ and $\delta$, which can be choosen uniform with respect to $m$, so that if for some $t> 0$ we have : \[ \sup_{x \in B_t}            \vert u_m(x) - l_{p_0}(x) \vert \leq \e t \quad l_{p_0} \text{ affine with } \nabla l_{p_0}=p_0 \] 
then, $\nabla u_m(B_{t\delta/2}) \subset \overline{B_{2\tilde{M}}} \setminus \mathcal{C}_m $ (the second case occurs beacuse $B_r(p_0)$ is outside the connected component $\mathcal{C}_m$ which contains $B_{\rho}(q)$). Hence, for any connected component $\mathcal{C}$  of \[\overline{B_{2 \tilde{M}}} \setminus \overline{\mathcal{N}^r(\mathcal{D}\cap \mathcal{S})}        \] and because $0$ is a Lebesgue point of $\nabla u$, there exists $t>0$ and $\delta>0$ independant of $m$ such that \[ \nabla u_m(B_{t\delta/2}) \subset        \overline{B_{2\tilde{M}}} \setminus \mathcal{C}_m. \] Since this number of connected component is finite (bounded by K, independant of $m$), this means that there exists $\tilde{\delta}>0$ independant of $m$ such that for all $m$ large enough, \[ \nabla u_m(B_{\tilde{\delta}}) \subset \overline{\mathcal{N}^r(\mathcal{D}\cap\mathcal{S})}              \]
The conclusion follows by letting $m \to \infty$.

\end{proof}

\begin{remark}
The Theorem shows exactly that $\nabla u $ is continuous in a neigborhood of $0$, provided $0$ is a Lebesgue point of $\nabla u $ with $\nabla u(0) \in \overline{B_M} \setminus \left( \mathcal{D}\cap \mathcal{S} \right)$. When this point belongs to the degenerate/singular set, we can't localise $\nabla u(B_{\delta})$ around it, but only on a neighborhood of the bad set. 
\end{remark}

\section{Blow-ups at non regular points}

Theorem \ref{Th inter} does not imply that the map $x \mapsto \dist(\nabla u(x),\mathcal{D}\cap\mathcal{S} ) $ has a continuous representative in the whole $B_1$. However, in view of the localisation Theorem \ref{p:apriori}, we expect that given a non-Lebesgue point $x_0$, the image $\nabla u(B_{\delta}(x_0)) $ should be close to $\mathcal{D}\cap \mathcal{S}$. \\

Until the end of this section, we assume that $u(0)=0$. Given $\delta >0$, we write $u_{\delta}$ the rescaling defined by $u_{\delta}(\cdot)= \delta^{-1} u(\delta \cdot) $. The sequence $(u_{\delta})_{\delta >0} $ is a sequence of Lipschitz maps with Lipschitz constant $M$, and therefore, we may extract a subsequence $  (u_{\delta_j})_{j\in \mathbb{N}} $ such that $(\nabla u_{\delta_j})_{j\in \mathbb{N}} $ generates a family of Young measures $(\nu_{x})_{x\in B_1}$, with the property that for any smooth function $H$ : 

\begin{equation}
\label{Young measure}
\int_{B_1} H(\nabla u_{\delta_j}(x)) dx \to \int_{B_1} \int_{B_M} H(y) d\nu_{x}(y) dx \quad \text{ as } j\to \infty
\end{equation}

 see for instance \cite[Theorem 3 and Remark 3]{Balder}. Our goal is to show : 
 
\begin{proposition}
\label{p:supportYoung}
Let $u$ a Lipschitz solution of \eqref{equation div}, with $G$ monotone. Assume that $u$ is not $C^1$ in any neighborhood of $0$. Then, $\nu_{x}$ is supported in $ \mathcal{D}$ for almost every $x$ in $B_1$. 
\end{proposition}  

A direct consequence is : 

\begin{Corollary}
Let $G$ continuous and strictly monotone, and  $u$ a solution of \eqref{equation div}. Then, for any $x_0 \in B_1$, either $x \mapsto \dist(\nabla u(x),\mathcal{D})$ is continuous at $x_0$, or it holds \[ \frac{1}{\vert B_{\delta}\vert} \int_{B_{\delta}(x_0)} \dist(\nabla u(y),\mathcal{D})dy \to 0 \text{ as } \delta \to 0.                 \] 
\end{Corollary}

We introduce our notations first, and prove  preliminary Lemmas, which will lead to the proof of Proposition \ref{p:supportYoung}. \\ For $\e>0$ and by compactness of $\overline{B_{2M}} \setminus \mathcal{N}^{\e}(\mathcal{D})$, it is sufficient to prove that for any ball $B_{\eta}(\xi) \subset \overline{B_{2M}} \setminus \mathcal{N}^{\e}(\mathcal{D})$, we have the support of $\nu_x$ is outside $B_{\eta}(\xi)$. \\
Until the end, we fix $\e>0$, and a ball $B_{\eta}(\xi) \subset\overline{B_{2M}} \setminus \mathcal{N}^{\e}(\mathcal{D})$.

We define $H$ a smooth function satisfying : 

\begin{align*}
 H(\zeta) = 
\begin{cases}
0 &\text{ if }\zeta\in \R^2 \setminus B_{\eta}(\xi),\\
1 &\text{ if }\zeta\in B_{\eta/2}(\xi).
\end{cases}
\end{align*}

We also require 
\begin{equation} 
\label{eq:H_kgeq5/8}
\lbrace H \geq 5/8 \rbrace \subset B_{3\eta/4}(\xi)                     \end{equation}
and 

\begin{equation} 
\label{eq:H_kleq3/4}
\lbrace H \leq 3/4 \rbrace \subset \R^2 \setminus B_{2\eta/3}(\xi).
\end{equation}
Observe that, the map $H$ satisfies : 
\begin{equation}
\label{eq:supportH_k}
\lbrace \nabla H \neq 0 \rbrace \subset B_{\eta}(\xi) \subset O_{\lambda}(G)
\end{equation}
where $ \lambda = \lambda(\e)$ is such that $\overline{B_{2M}}\setminus \mathcal{N}^{\e/2}(\mathcal{D}) \subset O_{\lambda}$. Recall $(u_m)_{m\geq 0}$ is the approximating sequence defined in section 2. Following  \cite[Lemma 2.4]{LacLam}, we have a uniform $W^{1,2}$ estimate on the sequence $H(\nabla u_m)$ : 

\begin{lemma}
\label{p:W22bound}
There exists $m_0 \in \mathbb{N}$ and $C(\e,H,G) >0$ such that for any $m\geq m_0$ : 
\[ \int_{B_{1/2}} \vert \nabla \left( H(\nabla u_m) \right) \vert^2 dx \leq C. \]
\end{lemma}

\begin{proof}
If $m_0$ is large enough (depending only on $\e$), by \eqref{eq:supportH_k} and Lemma \ref{l:approxG} we have \[ \lbrace \nabla H \neq 0 \rbrace \subset O_{\lambda}(G) \subset O_{\lambda/2}(G_m) \] for $\lambda(\e) >0$ given by \eqref{eq:supportH_k}, independant of $m$. Using the $W^{2,2}$ estimate (see \cite[Lemma 2.4]{LacLam}) on smooth solution $u$ in the set $\nabla u^{-1}(O_{\lambda})$, we have : 
\begin{align*}
\int_{B_{1/2}} \vert \nabla \left( H(\nabla u_m) \right) \vert^2 & \leq \int_{\nabla u_{m}^{-1}(O_{\lambda/2}(G_m)\cap B_{1/2}} \vert \nabla H(\nabla u_m) \vert^2 \vert \nabla^2 u_m \vert^2 dx \\
& \leq \| \nabla H \|_{L^{\infty}} \int_{\nabla u_{m}^{-1}(O_{\lambda/2}(G_m)\cap B_{1/2}} \vert \nabla^2 u_m \vert^2 dx \\
& \leq \|\nabla H \|_{L^{\infty}}\frac{c_0}{\lambda^2} \| G_m(\nabla u_m )\|_{L^2(B_1)}^2 \\
& \leq C(\e,H,G). 
\end{align*}

\end{proof}

We introduce, for $\delta >0 $ : 

\[ f(\delta) := \frac{\vert \lbrace H(\nabla u) \geq 3/4 \rbrace \cap B_{\delta} \vert}{\vert B_{\delta} \vert} = \frac{\vert \lbrace H(\nabla u_{\delta}) \geq 3/4 \rbrace \cap B_1 \vert}{\vert B_1 \vert}               . \]

The following result, which can be found in \cite[Lemma 5]{SV}, is crucial : 

\begin{lemma}
\label{p:SVscale}
Let $ v \in W^{1,2}(B_1)\cap L^{\infty}(B_1)$ with $M\geq v\geq 0$. Assume that there exists $\nu >0$ such that $ \vert \lbrace v \geq \frac{3M}{4} \rbrace \cap B_1 \vert \geq \nu \vert B_1 \vert$.
Then, one of the two situations occurs :
\begin{itemize}
\item Either \[\int_{B_1\setminus B_{\sqrt{\nu/2}}} \vert \nabla v \vert^2 \geq  M^2 \frac{\nu}{512 \pi^2} \]  \\
\item Or, there exists $s \in ( \sqrt{\nu/2},1) $ such that $ v \geq \frac{5M}{8} $ on $\partial B_{s}$.
\end{itemize} 

\end{lemma}
Next, we prove the following Lemma on the sequence $(u_m)_{m\in \mathbb{N}} $ of section 2  :

\begin{lemma}
\label{l:measureconvergence}
For any $\frac{1}{2} \geq \delta >0$, there exists $m:=m(\delta)$ such that for any $m\geq m(\delta)$ : 
\[ \frac{\vert \lbrace H(\nabla u_m) \geq 3/4 \rbrace \cap B_{\delta} \vert}{\vert B_{\delta} \vert}  \geq \frac{f(\delta)}{2}.                \]
\end{lemma}

\begin{proof}
This is a consequence of the uniform $W^{1,2}$ bound on the sequence $(H(\nabla u_m))_{m\in \mathbb{N}}$ : by Lemma \ref{p:W22bound}, the sequence $ (H(\nabla u_m))_{m\geq m_0}$ is bounded in $W^{1,2}(B_{1/2})$. Hence, up to extraction, the sequence $(H(\nabla u_m))_{m\in \mathbb{N}} $ converges strongly in $L^2(B_{1/2})$ toward $H(\nabla u)$ and the conclusion follows.
\end{proof}

As a consequence of the previous results, we get : 

\begin{proposition}
\label{p:subsequence}
Assume $u$ is not $C^1$ in any neighborhood of $0$. Then, for any sequence $\delta_j \to 0 $, there exists a subsequence $\delta_{j_i} \to 0 $ such that $f(\delta_{j_i}) \to 0 $ as $i\to \infty$.
\end{proposition}
\begin{proof}
Fix such a sequence $\delta_j$. One can assume $f(\delta_j) > 0 $ for all $j\in \mathbb{N}$ (if $f(\delta_0)=0$ for some $\delta_0 >0$, then $f(\delta)=0 $ for all $\delta \leq \delta_0$). We define the subsequence $\delta_{j_i}$ by \[\delta_{j_0} = \max \lbrace \delta_j \text{ such that } \delta_j \leq 1/2 \rbrace \text{ and } j_{i+1} = \min \lbrace j >j_i \vert \text{ } \delta_{j} < \delta_{j_i}\sqrt{f(\delta_{j_i})/4} \rbrace. \] Notice that the subsequence satisfies $\delta_{j_i} \leq 1/2 $ for any $i\geq 0$. \\ \\ We claim that for all $i \geq 0 $, there exists $m(i)$, such that for all $m\geq m(i)$, we have : 
\[ \int_{B_{\delta_{j_i}}\setminus B_{\delta_{j_{i+1}}}} \vert \nabla \left( H(\nabla u_m) \right) \vert^2 dx \geq \frac{ f(\delta_{j_i})}{\pi^21024}.                         \]

Fix $i\geq 0$. By Lemma \ref{l:measureconvergence}, there exists $m_0(i)$ such that for all $m\geq m_{0}(i)$, we have : 
\[  \frac{\vert \lbrace H(\nabla u_m) \geq 3/4 \rbrace \cap B_{\delta_{j_i}} \vert}{\vert B_{\delta_{j_i}} \vert}  \geq \frac{f(\delta_{j_i})}{2}                               . \]

For $m\geq m_{0}(i)$, we apply Lemma \ref{p:SVscale} (rescaled to $B_{\delta_{j_i}}$) to the non-negative function $ H(\nabla u_m) \in W^{1,2} \cap L^{\infty} $ ( with $W^{1,2}$ and $L^\infty$ bound independent of $m$ ). Hence : 
\begin{itemize}
\item Either  \[\int_{B_{\delta_{j_i}}\setminus B_{\delta_{j_{i+1}}}} \vert \nabla \left( H(\nabla u_m) \right) \vert^2 \geq \frac{f(\delta_{j_i})}{\pi^2 1024}.  \] \\
\item Or, there exists a radius $s \in (\delta_{j_{i+1}},\delta_{j_i})$  such that $H(\nabla u_m) \geq \frac{5}{8} $ on $\partial B_s$.
\end{itemize}

Now, if the second situation happens for $m \geq 0$, this means : 

\[ \nabla u_m(\partial B_s) \subset B_{3\eta/4}(\xi)                               \] 
because of \eqref{eq:H_kgeq5/8}. Applying the Hartman-Niremberg maximum principle (see \cite{HN59}) (as a consequence of $\det(\nabla^2 u_m) \leq 0 $ in $B_1$) :  \[ \partial \nabla u_m(B_s) \subset \nabla u_m(\partial B_s)\]  we deduce \[ \nabla u_{m}(B_{s}) \subset B_{3\eta/4}(\xi) \subset O_{\lambda}(G)  \] hence, 
\[ \nabla u_m(B_{\delta_{j_{i+1}}}) \subset O_{\lambda}(G).              \] If this situation happens for infinitely many's $m$, we can construct a subsequence  with this property and therefore obtain the inclusion $\nabla u(B_{\delta_{j_{i+1}}}) \subset O_{\lambda}(G)$, which is impossible because by \cite[Theorem 1.3]{LacLam}, this would mean that $u \in C^1(B_{{\delta_{j_{i+1}}}/2})$. Hence, the second case can not occur infinitely many times. Therefore, there exists $m(i)\geq m_{0}(i)$ such that for all $m\geq m(i) $ :
\[\int_{B_{\delta_{j_i}}\setminus B_{\delta_{j_{i+1}}}} \vert \nabla \left( H(\nabla u_m) \right) \vert^2 \geq \frac{f(\delta_{j_i})}{\pi^2 1024}.  \]To conclude the proof, we use the uniform $W^{1,2}$ bound on the sequence $H(\nabla u_m)$ : Fix $I \in \mathbb{N}$, and let $m \geq \max \lbrace m(i), i \in (0,...,I) \rbrace$. Because of \eqref{eq:supportH_k}, $\nabla H$ is zero in a neighborhood of $\mathcal{D}$ , thus, we can estimate, using Lemma \ref{p:W22bound} :

\begin{align*}
\frac{1}{\pi^2 1024} \sum_{i=0}^{I} f(\delta_{j_i}) & \leq \int_{B_{\delta_{j_0}} \setminus B_{\delta_{j_I}}} \vert \nabla (H(\nabla u_m)) \vert^2  dx \\
& \leq C.
\end{align*} Where we used that $B_{\delta_{j_0}} \subset B_{1/2}$. This means that the series on the left is convergent, and the conclusion follows.

\end{proof}

We are ready to prove Proposition \ref{p:supportYoung}. 

\begin{proof}[Proof of Proposition~\ref{p:supportYoung}]
Fix $\e >0$ and let us prove that $\nu_{x}$ is supported outside $B_{2\eta/3}(\xi)$. Since $u$ is not $C^1$ in any neighborhood of $0$,  we apply Proposition \ref{p:subsequence} to deduce that for any $\delta_{j} \to 0$, we can extract a subsequence $(\delta_{j_i})_{i\in \mathbb{N}}$ such that $f(\delta_{j_i}) \to 0 $ as $i\to \infty$. Let us define a smooth map $\tilde{H} : B_1 \to \R$ through \[\tilde{H} = 0 \text{ on } \lbrace H <3/4 \rbrace \] with $\tilde{H} \geq 0 $ and $\sup_{x\in \R^2}\vert \tilde{H} \vert \leq 1 $. \\

We estimate :
\begin{align*}
\int_{B_1} \tilde{H}(\nabla u_{\delta_{j_i}}(x))dx & \leq \left\vert \lbrace \tilde{H}(\nabla u_{\delta_{j_i}}) > 0 \rbrace \cap B_1 \right\vert \\
& \leq \left\vert  \lbrace H(\nabla u_{\delta_{j_i}}) \geq 3/4 \rbrace \cap B_1 \right\vert \\
& \leq \vert B_1 \vert  f(\delta_{j_i}) \to 0 \text{ as } i \to \infty.
\end{align*}

Thus, from \eqref{Young measure}, we deduce that \[ \int_{B_1} \int_{B_M} \tilde{H}(y) d\nu_x(y)dx = 0                  \] and therefore, for almost every $x$ in $B_1$, $\nu_{x}$ is supported in \[ \lbrace \tilde{H}=0 \rbrace \subset \lbrace H \leq 3/4 \rbrace \subset  \left(\R^2 \setminus B_{2\eta/3}(\xi) \right).   \] Since this is true for any ball $B_{\eta}(\xi) \subset \overline{B_{2M}} \setminus \mathcal{N}^{\e}(\mathcal{D})$, the proof is complete.
\end{proof}

At any point $x \in B_1$ where $u \in C^1$ in a neighborhood of $x$, we have \[            \frac{1}{\vert B_{\delta} \vert} \int_{B_{\delta}(x)} \vert \nabla u(y)-\nabla u(x) \vert dy \to 0  \text{ as } \delta \to 0    \] that is, $x$ is a Lebesgue point of $\nabla u$. A consequence of Proposition \ref{p:supportYoung} is :

\begin{Corollary}
\label{c:LPdistance}
Assume that $x \in B_1$ is not a Lebesgue point of $\nabla u$. Then : 
\[ \frac{1}{\vert B_{\delta} \vert} \int_{B_{\delta}(x)} \dist(\nabla u(y),\mathcal{D})dy \to 0 \text{ as } \delta \to 0   . \]
\end{Corollary}

\begin{remark}
We will use the duality argument presented in \cite{LacLam}, to deduce that the same holds replacing $\mathcal{D}$ by $\mathcal{S}$. This explains why we stated Theorem \ref{Th inter} with $0$ being a Lebesgue point of $\nabla u$ instead of $0$ being a differentiability point of $u$ : a Lebesgue point of $\nabla u $ is a Lebesgue point of $G(\nabla u) = i\nabla v$ and the converse is true thanks to the properties of $G$ (see below). It is not directly obvious that the differentiablity of $u$ at a point implies the one of the dual map $v$.
\end{remark}

Recall that $i$ is the counter clock-wise rotation of angle $\pi/2$. Following \cite[Proposition 4.1]{LacLam}, for any Lipschitz solution of \eqref{equation div}, there is a strictly monotone vector field $G^*$ and a map $v : B_1 \to \R$ such that $-i\nabla v = G(\nabla u) $ and $v$ solves $\dv(G^*(\nabla v))=0$. The relation between the degenerate/singular sets of $G$ and $G^*$ is given by  \[ (iG)(\mathcal{S}(G)) = \mathcal{D}(G^*) \] We can use Corollary \ref{c:LPdistance} to obtain, provided $x$ is not a Lebesgue point of $\nabla v = iG(\nabla u)$ :

\begin{align*}
\frac{1}{\vert B_{\delta}\vert}\int_{B_{\delta}(x)} \dist(\nabla v(y),\mathcal{D}(G^*))dy & = \frac{1}{\vert B_{\delta}\vert}\int_{B_{\delta}(x)} \dist(iG(\nabla u)(y),\mathcal{D}(G^*))dy \\
& = \frac{1}{\vert B_{\delta}\vert}\int_{B_{\delta}(x)} \dist(iG(\nabla u)(y),iG(\mathcal{S}(G)))dy \\
& \to 0 \text{ as } \delta \to 0,
\end{align*}
which implies, by injectivity and continuity of $iG$ : 
\[\frac{1}{\vert B_{\delta}\vert}\int_{B_{\delta}(x)} \dist(\nabla u(y),\mathcal{S}(G))dy  \to 0 \text{ as } \delta \to 0. \]  This computations suggests that we can replace $\mathcal{D} $ by $\mathcal{S}$ in the case of strictly monotone fields. \\

There is a natural relation between the Lebesgue points of $\nabla u$ and those of $iG(\nabla u)$ : 

\begin{lemma}
\label{l:LPDuGDu}
$x \in B_1$ is a Lebesgue point of $\nabla u$ if and only if $x$ is a Lebesgue point of $iG(\nabla u)$.
\end{lemma}

\begin{proof}
This is a direct consequence of the fact that $G$ is an homeomorphism.
\end{proof}

Using this correspondence, we show a better version of Corollary \ref{c:LPdistance}, which is exactly Theorem \ref{T: LooseTheorem}.

\begin{theorem}
Let $G:\R^2 \to \R^2$ continuous and strictly monotone, and let $u$ a Lipschitz solution of $\dv(G(\nabla u))=0$. Then, for any $x\in B_1$, either $y \mapsto \dist(\nabla u(y),\mathcal{D}\cap\mathcal{S})$ is continuous at $x$, or it holds :
\[ \frac{1}{\vert B_{\delta}\vert}\int_{B_{\delta}(x)} \dist(\nabla u(y),\mathcal{D}\cap\mathcal{S}) dy \to 0 \text{ as } \delta \to 0 . \]
\end{theorem}
\begin{proof}
By Theorem \ref{Th inter}, if $\xi \mapsto \dist(\nabla u(\xi),\mathcal{D}\cap\mathcal{S})$ is not continuous at $x$, then, $x$ is not a Lebesgue point of $\nabla u$. 	Hence, by Lemma \ref{l:LPDuGDu}, $x$ is not a Lebesgue point of $iG(\nabla u)$. Applying Corollary \ref{c:LPdistance}, we get 
\begin{align*}
\frac{1}{\vert B_{\delta} \vert} \int_{B_{\delta}(x)} \dist(\nabla u(y),\mathcal{D})dy & \to 0 \text{ as } \delta \to 0, \\
\frac{1}{\vert B_{\delta} \vert} \int_{B_{\delta}(x)} \dist(\nabla u(y),\mathcal{S})dy & \to 0 \text{ as } \delta \to 0 
\end{align*}
This implies that the Young measure generated by the sequence $(u_{\delta})_{\delta>0}$ is supported in $\mathcal{D}\cap\mathcal{S}$ and the conclusion follows.
\end{proof}

\section{Appendix}

\subsection{Appendix A : Proof of Lemmas \ref{Lemme inter vide}}

Recall Lemma \ref{Lemme inter vide} :
\begin{lemma}
Let $G$ smooth strongly monotone, and $u$ a smooth solution of $\dv(G(\nabla u))=0$. 
Let $p,q $ and $\rho >0 $ such that $B_{\rho}(q) \subset O_{\lambda}\cap V_{\Lambda} $ and $p_0 \notin B_{\rho}(q)$. Then, there exists $\e :=\e(\lambda,\Lambda,\rho)$ such that if 
\[ \vert u - \langle p_0, x \rangle \vert \leq \e \text{ for all } x \in B_1            \]
then, $\nabla u(B_{1/2}) \cap B_{\rho/2}(q)=\emptyset$.
\end{lemma}

\begin{remark}
\begin{itemize}
\item The lemma is stated in the particular case $l_p = \langle p,x \rangle$. It is sufficient to prove it in this setting : indeed, if $l_p = \langle p,x \rangle + c $ for some constant $c$, then, define $\tilde{u}=u-c$. Then, $\tilde{u}$ is a solution of the same equation and satisfies the flatness assumption as stated in the Lemma. Since $\nabla \tilde{u}=\nabla u $ the conclusion follows.
\item As it will be clear in the proof, one needs only the inclusion $B_{\rho}(q) \subset O_{\lambda}\cap \tilde{V}_{\Lambda} $ where \[ \tilde{V}_{\Lambda} := \interior \lbrace \xi \in \R^2, \limsup_{\zeta \to 0} \frac{\langle G(\xi+\zeta)-G(\xi),\zeta \rangle}{\vert \zeta \vert^2} \leq \Lambda \rbrace      . \] It is shown in \cite{LacLam} that we have the inclusion $V_{\Lambda} \subset \tilde{V}_{\Lambda}$ for any continuous, strictly monotone field $G$.
\end{itemize}
\end{remark}

Since the proof of Lemma \ref{Lemme inter vide} in \cite{DSS} is embedded in \cite[Proposition 6.2]{DSS}, we rewrite it for readers convenience. It based on the following geometrical result, which can also be found in \cite[Theorem 5.3]{DSS} : 

\begin{proposition}
\label{PropDSS}
Let $u \in \mathcal{C}^{\infty}(B_1)$ a solution of $a_{ij}(x)u_{ij} = 0 $  in $B_1$ with $A(x) := a_{ij}(x)$ positive definite. Assume that $u$ is not linear. Then, in each open set $U$, there is a point $x_U$ such that each set \[ \lbrace u> l_U  \rbrace , \lbrace u<l_U \rbrace            \] where \[l_U : x \mapsto  u(x_U)+\langle x-x_U, \nabla u (x_U) \rangle \] has at least two distinct connected components in $B_1$ that intersects any neighborhood of $x_U$. Moreover, this connected components are not compactly supported in $B_1$.

\end{proposition}

This result is purely $2$ dimensionnal, it is a consequence of the maximum principle and the fact that $\nabla^2 u$ has eigenvalues with opposite signs.

Here is the proof of Lemma \ref{Lemme inter vide} :

\begin{proof}[Proof of Lemma~\ref{Lemme inter vide}]

We assume that $u$ is not linear, otherwise the result is obvious. 
Assume by contradiction that there exists $x_0 \in B_{1/2}$ such that $\nabla u(x_0) = p_1 \in B_{\rho/2}(q)$. Since $u$ is not linear, we can apply Propositon \ref{PropDSS} with $U$ any neighborhood of $x_0$. Without loss of generality, we may assume that $x_0$ is the point $x_U$ of Proposition \ref{PropDSS} (otherwise, take a smaller neighborhood). Thus, we get that each set
\begin{align*}
& \lbrace u(x) < l(x) := u(x_0)+p_1 \cdot (x-x_0) \rbrace, \\ & \lbrace u(x) > l(x) := u(x_0)+p_1 \cdot (x-x_0) \rbrace 
\end{align*} 
has in $B_1$ at least two distinct connected components that intersect any neighborhood of $x_0$. \\

On the other hand, the flatness assumption gives 

\[ \lbrace  \langle x-x_0,p_0-p_1 \rangle < -2 \e \rbrace \subset \lbrace u<l  \rbrace                  \] and \[  \lbrace  \langle x-x_0,p_0-p_1 \rangle > 2 \e \rbrace \subset \lbrace u>l  \rbrace.           \] This means that the set $ \lbrace u<l \rbrace $ has one connected component included in the strip \[ \lbrace \vert \langle x-x_0, p_0-p_1 \rangle \vert \leq 2\e \rbrace.               \] Changing the coordinates in the $x$ and $p$-space, we may assume that \[p_1=0, \quad p_0= \alpha e_2,\quad x_0 =  e_1 /4, \quad  u(x_0)=0, \text{ with } M \geq \alpha> \rho/2 \] because $p_0 \notin B_{\rho}(q)$, hence $p_0 \notin B_{\rho/2}(p_1)$. Recall that $B_{\rho}(q) \subset O_{\lambda}\cap \tilde{V}_{\Lambda}$, we assume for simplicity that $ \lambda \leq 1/\Lambda$ and therefore, we have $B_{\rho}(q) \subset O_{\lambda}\cap \tilde{V}_{1/\lambda}$.\\
Define the rectangle \[ \overline{R} := \lbrace \vert x_1 \vert \leq 1/8, \vert x_2 \vert \leq 1/8 \rbrace.                     \] Without loss of generality, we can assume that the connected component of $\lbrace u<0 \rbrace $ in the thin strip $\lbrace \vert \langle x,e_2\rangle \vert \leq 2\e \rbrace $ goes out of $B_1$ by the left side of $x_0$ and $ u< 0 $ below the horizontal line $\lbrace x_2 = -2 \e \rbrace$.   Therefore $ \lbrace u<0 \rbrace$ has in $\overline{R}$ a connected component included in the strip $\lbrace \vert \langle x,\alpha e_2 \rangle \vert \leq 2\e \rbrace $ that intersects both $ x_1 = \pm 1/8$ and we let $U$ be the connected component of $\lbrace u<0 \rbrace $ in $\overline{R}$ which contains $ \lbrace x_2=-1/8 \rbrace$. \\ 
The goal is to construct a subsolution $w$ of the PDE that is $\tr(\nabla G(\nabla w)D^2w) > 0$, increasing in the $x_2$ direction, with the property that $u-w > 0 $ on the lateral sides of the set $\overline{R}\setminus U$ and $u-w < 0 $ in $\partial U \cap \lbrace x_1=0 \rbrace$. If we can do so, by the minimum principle, this means that the minimum of $u-w$ occurs at some point $z_0 \in \partial U \cap R $ and since $w$ is increasing in the $x_2$ direction, we infer $u(x)> w(x)-w(z_0) > 0 $ for $x$ in the line $z_0 +t e_2 $, $t>0$. This is a contradiction because this line has to intersect the other connected component of $\lbrace u< 0 \rbrace$. Now, we construct such a $w$. The idea is to use the fact that $B_{\rho/2}(0) $ is included in the elliptic region, and therefore, construct $w$ with small gradient to obtain a sub-solution. \\

 For $\gamma$ and $k$ to be choosen, define the functions \[v = x_2-20x_1^2 \quad w := \gamma \exp(kv)-c \] where $c$ such that $w(-1/8)=0$. We have \[ \nabla^2 w =  \gamma k \exp(k v) \left( \nabla^2 v + k \nabla v \otimes \nabla v \right).          \] For any $k\geq 0$, we can choose $\gamma$ small so that $\nabla w \in B_{\rho/2}$. Since $B_{\rho/2}(0) \subset O_{\lambda}\cap \tilde{V}_{1/\lambda}$, denoting $ \nabla^2 w^{\pm}  $ the positive and negative part of the Hessian of $w$, and using the inequality \[ \Tr(\nabla^s G(\nabla w)\nabla^2 w ) \geq \lambda \vert \nabla^2 w ^+ \vert - \frac{1}{\lambda} \vert \nabla^2 w ^- \vert      \] we can check that for $k$ large enough depending on $\lambda$, the function $w$ satisfies $ \Tr(\nabla^s G(\nabla w)\nabla^2 w) > 0 $. Indeed, for \[ k \geq \left(\frac{80+\sqrt{6400+4000\lambda^2}}{10 \lambda} \right)^2     \] we compute :
\begin{align*}
\nabla ^2 w = \gamma k \exp(kv) \begin{pmatrix}
1600kx_1^2-40 & -40kx_1 \\ -40kx_1 & k
\end{pmatrix}.
\end{align*}
We seek for the negative/positive part of the matrix  \[A :=\begin{pmatrix}
1600kx_1^2-40 & -40kx_1 \\ -40kx_1 & k
\end{pmatrix} = A^+ + A^-. \] We have : 
\begin{align*}
\lambda \vert A^+ \vert + \frac{1}{\lambda} \vert A^- \vert = \frac{40k \lambda}{\vert A^- \vert}-\frac{1}{\lambda}\vert A^- \vert
\end{align*} which is $ > 0 $ provided $ \vert A^- \vert \leq \lambda\sqrt{40k}$. Now, we use the characteristic polynomial for a $2\times2$ matrix to obtain \[ \vert A^- \vert = \frac{80k}{1600kx_1^2-40 + k +\sqrt{\left( 1600kx_1^2-40 + k \right)^2 + 40k}}.      \] It remains to see that for $k$ as above, we actually have \[ \vert A^- \vert \leq \lambda \sqrt{40k}.    \] For this $k$, we then define \[\gamma := \frac{\rho}{4}\frac{\exp(-k)}{\sqrt{1601}k}  \] so we have $\nabla w(B_1) \subset B_{\rho/2}(0)$, and therefore, $w$ is a subsolution of the PDE, so $u-w$ cannot achieve its minimum in the interior of $\overline{R}\setminus U$.

Now we compare $u$ with $w$, notice that :
\[ \alpha \langle e_2,x \rangle > \frac{\rho}{16} \quad \text{ on } \lbrace x_2= 1/8, \vert x_1 \vert \leq 1/8 \rbrace  \] and  for this $\gamma$, we have\[ w= \gamma \exp(k\left(1/8-20 x_1^2 \right)) - \gamma \exp(-k/8) < \rho/32  \quad \text{ on } \lbrace x_2= 1/8, \vert x_1 \vert \leq 1/8 \rbrace.  \]
Therefore, \[  \alpha \langle e_2,x \rangle - w > \frac{\rho}{32} \quad \text{ on } \lbrace x_2= 1/8, \vert x_1 \vert \leq 1/8 \rbrace        \]

We do the same comparison on the lateral sides of the rectangle. For $\theta > 0 $ small, we have : 
\[ \alpha \langle e_2,x \rangle > -M \theta             \quad \text{ on } \lbrace x_1 = \pm 1/8 , x_2 \geq -\theta \rbrace \] and  in $ \lbrace x_1 = \pm 1/8 , x_2 \geq -\theta \rbrace $ we have : 
\begin{align*}
w & = \gamma \left( \exp(k(x_2-20 x_1^2)) - \exp(-k/8) \right) \\
& \leq \gamma \left( \exp(- \frac{12k}{64}) - \exp(-k/8) \right) \\
& \leq - \gamma \exp(-k/8) \left( 1- \exp(\frac{-k}{16}) \right) \\
& \leq - \frac{\gamma \exp(-k/8)}{2}.
\end{align*}

This means that : 
\begin{align*}
\alpha \langle e_2,x \rangle - w & > -M \theta + \frac{\gamma \exp(-k/8)}{2} \\
& > \frac{\gamma \exp(-k/8)}{4} \\
& > 0
\end{align*}
for $ \theta < \frac{\gamma \exp(-k/8)}{4M} $ in $ \lbrace x_1 = \pm 1/8 , x_2 \geq -\theta \rbrace $.

It is time to set $\e$. We let $ \e := \frac{\gamma \rho \exp(-k/8)}{8M}$  and we use the flatness assumption to get:
\begin{align*}
u-w & \geq \alpha \langle e_2,x \rangle - \e - w \\
& \geq \frac{\gamma \exp(-k/8)}{4}- \frac{\gamma \rho \exp(-k/8)}{8M} \\
& > 0 \quad \text{ on } \lbrace x_1 = \pm 1/8, x_2 \geq -2 \e /\rho \rbrace
\end{align*}

and 
\begin{align*}
u-w & \geq \alpha \langle e_2,x \rangle - \e - w \\
& \geq \frac{\rho}{32}- \frac{\gamma \rho \exp(-k/8)}{8M} \\
& > 0 \quad \text{ on } \lbrace \vert x_1\vert \leq  1/8, x_2 = 1/8 \rbrace.
\end{align*} However, $u-w<0$ on $\lbrace x_1=0 \rbrace \cap \partial U$, hence the minimum of $u-w$ occurs at some point $z_0 \in \partial U \cap R$. This means that \[ u \geq w-w(z_0) \text{ in } \overline{R}\setminus U.          \] This is a contradiction, since one the line $z_0+te_2$ the function $w$ is increasing, meaning that for $t_0$ such that $z_0+t_0e_2$ belongs to the other connected component of $\lbrace u<0 \rbrace$ in $\overline{R}$ we have $u(z_0+t_0e_2)>0$.
 
 Therefore, for  \[\e = \frac{\rho^2 \exp(-k)}{32Mk\sqrt{1601}} \exp(-k/8) \quad k = \left(\frac{80+\sqrt{6400+4000\lambda^2}}{10 \lambda} \right)^2   \] we obtain what we want. 
\end{proof}

\subsection{Appendix B : Proof of Theorem \ref{p:apriori}}

The proof is based on two localisations Lemma. The first one was obtained in the variational setting in \cite{DSS}, the second one was obtained in \cite{LacLam}.

\begin{lemma}\label{l:locOlambda}
Let $u$ a solution of  $\dv G(\nabla u)=0 $ in $B_1$ and assume that  
\[ \nabla u (B_1) \cap B_{\rho}(\xi_0) = \emptyset \text{ } \text{ and } \text{ } B_{4 \rho}(\xi_0)  \subset O_{\lambda}(G)\, , \]
for some $\lambda,\rho > 0$ and $\xi_0\in\R^2$.
Then we have
\begin{align*}
\text{either }
&
\nabla u(B_{\delta}) \subset B_{4 \rho}(\xi_0), 
%\\
\quad
\text{or }
%&
\nabla u (B_{\delta}) \cap B_{3 \rho}(\xi_0) = \emptyset \, ,
\end{align*}
for some $\delta>0$ depending on $\lambda$, $\rho$, and $\| G(\nabla u) \|_{L^2(B_1)}$.
\end{lemma}

\begin{lemma}\label{l:locVLambda}
Let $u$ a solution of  $\dv(G(\nabla u))=0 $ in $B_1$ and assume that  
\[ \nabla u (B_1) \cap B_{\rho}(\xi_0) = \emptyset \text{ } \text{ and } \text{ } B_{4 \rho}(\xi_0)  \subset V_{\Lambda}(G)\, , \]
for some $\Lambda,\rho > 0$ and $\xi_0\in\R^2$.
Then we have
\begin{align*}
\text{either }
&
\nabla u(B_{\delta}) \subset B_{4 \rho}(\xi_0), 
%\\
\quad
\text{or }
%&
\nabla u (B_{\delta}) \cap B_{3 \rho}(\xi_0) = \emptyset \, ,
\end{align*}
for some $\delta>0$ depending on $\Lambda$, $\rho$, $\| \nabla u \|_{L^2(B_1)}$, 
and any $c>0$ such that $\omega_G(t)/t\geq c$ for $t\in [\rho, \|\nabla u\|_\infty+3\rho ] $.
\end{lemma}

\begin{proof}[Proof of Theorem \ref{p:apriori}]

Let $u$ be a smooth  solution of $\dv G(\nabla u)=0 $
 in $B_1$ with $|\nabla u|\leq M$.
By assumption we have
\begin{align*} 
\overline B_{2M}\subset  V_{\Lambda}(G) \cup O_{\lambda}(G)\cup U
\, ,
\end{align*} 
and we fix a Lebesgue number $\eta\in (0,\rho))$ of this open covering, with the property that
any ball $B_\eta(\xi)$ centered at $\xi\in \overline B_{2M}$ is contained in 
$V_{\Lambda}(G)$, $O_{\lambda}(G)$, or U.
We set $\e=\eta/4$.
The ball $\overline B_{2M}$ can be covered by a finite number of balls of radius $\e$.
Removing the balls that are contained in $U$ we are left with a covering
\begin{align*}
\overline B_{2M}\setminus U \subset \bigcup_{k=1}^K B^k_{\e},
\end{align*}
with $K\leq c M^2/\eta^2$ for some universal constant $c>0$, and the property that each ball $B^k_{4\e}$ satisfies
\begin{align*}
B^k_{4\e}\subset V_\Lambda(G)
\quad
\text{or}
\quad 
B^k_{4\e}\subset O_\lambda(G).
\end{align*}
By assumption, there exists a ball $B_{\rho}(q)$ included in a connected component of $ \bigcup_{k=1}^K B^k_{\e},$ such that $ \nabla u(B_1) \cap B_{\rho}(q)=\emptyset $. Hence, there exists $k \in \lbrace 1,,,K \rbrace$ such that
\[ \nabla u (B_1) \cap B_{\e}^k = \emptyset           \]
Since
 $ B^k_{4\e} \subset V_{\Lambda}(G)$ or $B^k_{4\e} \subset O_{\lambda}(G) $, 
 we can apply Lemma~\ref{l:locVLambda} or Lemma~\ref{l:locOlambda} 
  to ensure the existence of some $\delta>0$ such that
\begin{align*}
\text{either }\nabla u(B_\delta)\subset B_{4\e}^k
\quad\text{or }
\nabla u(B_\delta)\cap B_{3\e}^k=\emptyset.
\end{align*}  
If the first case occurs, then we are done since $4\e=\eta$. 
If the second case occurs, 
we infer that 
$ \nabla u(B_{\delta}) \cap B^j_{\e} = \emptyset $ for all neighboring balls $B_{\e}^j$ such that $ B_{\e}^j \cap B_{\e}^k \neq \emptyset $. 
Then we can apply again Lemma~\ref{l:locVLambda} or Lemma~\ref{l:locVLambda}
to the rescaled function $\delta^{-1}u(\delta \cdot)$ 
and these neighboring balls $B_{\e}^j$. 

We iterate this argument: if at some step we reach the first case, we are done.
Otherwise, we eventually covered the connected component $ \mathcal{C}$ of $ \overline{B_{2M}}\setminus U$ containing the ball $B_{\rho}(q)$ with the neighboring balls added at each step, and deduce that
$\nabla u(B_{\delta'})\subset \overline{B_{2M}}\setminus \mathcal{C}$.
Here $\delta'=\delta^K$ for $\delta$ as in Lemma~\ref{l:locVLambda} and Lemma~\ref{l:locOlambda}, and $K$ the number of iterations. This concludes the proof

\end{proof}

\bibliographystyle{acm}
\bibliography{ref}

\end{document}